\documentclass[a4paper,reqno,12pt]{amsart}

\usepackage{amssymb,amsthm,amsmath}
\usepackage{xcolor}
\usepackage{hyperref}
\hypersetup{
	colorlinks=true,
	linkcolor=blue,
	citecolor=red,
}
\usepackage[left=3cm,right=2.7cm,top=3cm,bottom=2.7cm,twoside]{geometry}

\newtheorem{theorem}{Theorem}[section]

\newtheorem{corollary}[theorem] {Corollary}
\newtheorem{defin}[theorem]{Definition}
\newenvironment{definition}{\begin{defin}\rm}{\end{defin}}
\newtheorem{ex}[theorem]{Example}
\newenvironment{example}{\begin{ex}\rm}{\end{ex}}
\newtheorem{lemma} [theorem]{Lemma}

\newtheorem{proposition}[theorem]{Proposition}
\newtheorem{rem}[theorem]{Remark}
\newenvironment{remark}{\begin{rem}\rm}{\end{rem}}

\newcommand{\N}{\mathbb{N}}
\newcommand{\D}{\mathcal{D}}

\newtheorem{thm}{Theorem}

\numberwithin{equation}{section}

\makeatletter
\@namedef{subjclassname@2020}{%
	\textup{2020} Mathematics Subject Classification}
\makeatother

\begin{document}

\title[Convergence of sub-series' and sub-signed series']{Convergence of sub-series' and sub-signed series'\\[5pt] in terms of  the asymptotic $\psi$-density}

\author{J.~Heittokangas and Z.~Latreuch}

\begin{abstract}
Given a non-negative real sequence $\{c_n\}_n$ such that the series $\sum_{n=1}^{\infty}c_n$ diverges, it is known that the size of an infinite subset $A\subset\N$ can be measured in terms of the linear density such that the sub-series $\sum_{n\in A}c_n$ either (a) converges or (b) still diverges. The purpose of this research is to study these convergence/divergence questions by measuring the size of the set $A\subset\N$ in a more precise way in terms of the recently introduced asymptotic $\psi$-density. The convergence of the associated sub-signed series $\sum_{n=1 }^{\infty}m_nc_n$ is also discussed, where $\{m_n\}_n$ is a real sequence with values restricted to the set $\{-1, 0, 1\}$.


\end{abstract}

\maketitle

\renewcommand{\thefootnote}{}
\footnotetext[1]{Corresponding author: J.~Heittokangas.}
\footnotetext[2]{Z.~Latreuch was supported by  the National Higher School of Mathematics. He would like to thank the Department of Physics and Mathematics at the University of Eastern Finland for its hospitality during his visit in May 2024.}


\section{Introduction \& motivation}

In this research, a typical starting point is a divergent series $\sum_{n=1}^{\infty}c_n$, where $\{c_n\}_n$ is a non-negative real sequence. If $\{m_n\}_n$ is a real sequence with values restricted to the set $\{-1, 0, 1\}$, we say that the associated series $\sum_{n=1 }^{\infty}m_nc_n$ is a
\begin{itemize}
\item \emph{sub-series}, if $m_n\in \{0,1\}$;
\item \emph{signed series}, if $m_n\in\{-1,1\}$;
\item \emph{sub-signed series}, if $m_n\in\{-1,0,1\}$.
\end{itemize}

A particularly useful way to express a given sub-series is in terms of the characteristic function $\chi_A$ of a set $A\subset\N$, in which case the sub-series takes the form
	\begin{equation}\label{subseries}
	\sum_{n=1}^\infty \chi_A(n)c_n.
	\end{equation}
For example, the harmonic series with $c_n=1/n$ diverges, but choosing $A\subset\N$ to be the set of all squared positive integers, the sub-series \eqref{subseries} converges.
Typical examples of signed series' are the alternating series' such as the alternating harmonic series $\sum_{n=1}^\infty \frac{(-1)^{n+1}}{n}$, which are conditionally convergent. The concept of sub-signed series seems to be new, although such series' have been studied in the literature \cite{GM}. A comprehensive collection of results on divergent series' and alternating series' can be found in Hardy's monograph \cite{H}.

The \emph{lower} and \emph{upper linear densities} of a set $A\subset\N$ are given, respectively, by
	$$
	\underline{\operatorname{d}}(A)=\liminf_{n\to\infty}\frac{A(n)}{n}
	\quad\textnormal{and}\quad
	\overline{\operatorname{d}}(A)=\limsup_{n\to\infty}\frac{A(n)}{n},
	$$
where $A(n):=\sum_{k=1}^n\chi_A(k)$. In particular, if $A=\{\varphi(n):n\in \N\}$, where \(\{\varphi(n)\}_n\) is a strictly increasing sequence of non-negative integers, then according to \cite{P}, the lower and upper linear densities of $A$ can be expressed as
	\begin{eqnarray}\label{LUdens}
	\underline{\operatorname{d}}(A)=\liminf_{n\to\infty}\frac{n}{\varphi(n)}
\quad\textnormal{and}\quad
\overline{\operatorname{d}}(A)=\limsup_{n\to\infty}\frac{n}{\varphi(n)}.
	\end{eqnarray}
It is clear that $0\leq \underline{\operatorname{d}}(A)\leq \overline{\operatorname{d}}(A)\leq 1$. If $\underline{\operatorname{d}}(A)=\overline{\operatorname{d}}(A)$, their common value $\operatorname{d}(A)$ is called the \emph{linear density} of $A$. 

The convergence and divergence of the sub-series \eqref{subseries} are closely related to the linear density of the set \(A\). One of the earliest findings in this regard is the following result by Auerbach from 1930 \cite{Au}, which, according to Footnote 1 in \cite{Au}, was actually obtained in 1923.

\begin{thm}\label{Aue}
\textnormal{\cite[Hilfsatz]{Au}}
Let $\{c_n\}_n$ be a sequence of non-negative real numbers. If $\sum_{n=1}^\infty c_n$ diverges, then there exists an infinite set $A \subset \N$ with $\operatorname{d}(A)=0$ such that \eqref{subseries} also diverges.
\end{thm}

We note that Theorem~\ref{Aue} has been independently rediscovered in subsequent works \cite{EK,NS}, and, in a stronger form involving lacunary sub-series in \cite{Ag,SF}.
An example illustrating Theorem~\ref{Aue} is the harmonic series, where one can consider $A$ to be the set of prime numbers, which has a null density (by the Prime number theorem) and for which \eqref{subseries} diverges by \cite[Theorem~1.13]{Apostol2}.

Let $\mathcal{I}$ denote the class of non-increasing, non-negative real sequences $\left\{c_n\right\}_n$ that satisfy the condition
	\begin{equation}\label{Salat-condition30}
	\sum_{n=1}^\infty c_n=\infty.
	\end{equation} 
In 1947, Hamming \cite{Hamming,Hamming2} stated that for a strictly increasing sequence of non-negative integers $\{\varphi(n)\}_n$, the series $\sum_{n=1}^{\infty} c_{\varphi(n)}$ diverges whenever $\left\{c_n\right\}_n \in \mathcal{I}$, if and only if
\begin{equation}\label{Ham_cond}
	\limsup_{n \to \infty} \frac{\varphi(n)}{n} < \infty.
\end{equation}
Consider the set $A=\{\varphi(n): n \in \mathbb{N}\}$, where $\{\varphi(n)\}_n$ is an arbitrary sequence defined as above. It follows that $A \subset \mathbb{N}$, and from \eqref{LUdens} and \eqref{Ham_cond}, we have
	$$
	\underline{\operatorname{d}}(A) = \liminf_{n \to \infty} \frac{n}{\varphi(n)} > 0.
	$$
Consequently, Hamming's theorem can be reformulated as follows.

\begin{thm}\label{New-Hamm}
\textnormal{\cite{Hamming2}}
Let \(A \subset \mathbb{N}\) be an infinite set. Then \(\underline{\operatorname{d}}(A) = 0\) if and only if there exists \(\{c_n\}_n \in \mathcal{I}\) such that 	
	\begin{equation}\label{convergent-sum0}
	\sum_{n=1}^\infty \chi_A(n)c_n<\infty.
	\end{equation}
\end{thm}

In his proof, Hamming showed only the existence of $\{c_n\}_n \in \mathcal{I}$ that satisfies \eqref{convergent-sum0} when $\underline{\operatorname{d}}(A) = 0$, leaving the converse implication to the reader; which was later proved by \v{S}al\'at in \cite[Theorem~2]{Salat}. Note that, in general, we cannot replace $\underline{\mathrm{d}}(A)=0$ with $\mathrm{d}(A)=0$ in the converse  implication of Theorem~\ref{New-Hamm}. Indeed, for any $\delta \in (0,1]$, there always exists a subset $A \subset \N$ such that $\underline{\operatorname{d}}(A) = 0\) and \(\overline{\operatorname{d}}(A) = \delta$ \cite{LP}. This, together with the first implication of Theorem~\ref{New-Hamm}, leads to the following proposition.

\begin{proposition}\label{Salat-example}
For every $\delta \in \left(0, 1\right]$, there exists a set $A \subset \N$ and a real sequence $\{c_n\}_n$ with the following properties: 
	$$
	\overline{\operatorname{d}}(A) = \delta, \quad
	\underline{\operatorname{d}}(A) = 0, \quad
	\sum_{n=1}^\infty c_n = \infty, 
\quad
	\sum_{n=1}^\infty \chi_A(n)c_n < \infty.
	$$
\end{proposition}

In 1958, Moser \cite{Moser} proved that if $A\subset\N$ satisfies $\sum_{n=1}^\infty \frac{\chi_A(n)}{n} < \infty$, then $\operatorname{d}(A) = 0$. In 1964, \v{S}al\'at \cite{Salat} extended this result to a broader class of sequences. Indeed, he showed that this convergence phenomenon holds for certain non-increasing, non-negative sequences $\{c_n\}_n$ that tend to zero.

\begin{thm}\label{Salat-thm1}
	\textnormal{\cite[Theorem~1]{Salat}}
	Let $\{c_n\}_n$ be a non-negative non-increasing sequence tending to zero such that 
	\begin{equation}\label{Salat-condition}
		\liminf_{n\to\infty}nc_n>0.
	\end{equation}
If $A\subset\N$ is an infinite set such that
\eqref{convergent-sum0} holds, then $\operatorname{d}(A)=0$.
\end{thm}

By the divergence of the harmonic series, the condition \eqref{Salat-condition} for $\{c_n\}_n$ implies \eqref{Salat-condition30}. 
Conversely, \eqref{Salat-condition30} does not generally imply \eqref{Salat-condition}.  
As shown by Proposition~\ref{Salat-example}, the assumption \eqref{Salat-condition} in Theorem~\ref{Salat-thm1} cannot be replaced with the weaker condition \eqref{Salat-condition30}.

\begin{remark}
A particular version of Proposition~\ref{Salat-example}, where $\delta \in (0,1)$ and $\overline{\operatorname{d}}(A) \geq \delta$, was essentially proved by \v{S}al\'at  \cite[pp.~211--212]{Salat} using a different method. However, since the notation and terminology in \cite{Salat} are very different from ours, and since many details are missing in \cite{Salat}, a complete proof is given in Appendix~\ref{App-A} for the convenience of the reader.
\end{remark}

We proceed to deal with sub-signed series' derived from the series $\sum_{n=1 }^{\infty}c_n$, where $\{c_n\}_n$ is a non-negative real sequence. A sub-signed series is represented by 
$
	\sum_{n=1 }^{\infty}m_nc_n,
$
where $\{m_n\}_n$ is a real sequence with values restricted to $\{-1, 0, 1\}$. This can be expressed equivalently as
	\begin{eqnarray}\label{Sub-signed}
	\sum_{n=1 }^{\infty}(\chi_{A}(n)-\chi_B(n))c_n,
	\end{eqnarray}
where
	\begin{eqnarray}\label{Sets}
	A:=\{n\in \mathbb{N}: m_n=1\} \quad \text{and} \quad B:=\{n\in \mathbb{N}: m_n=-1\}.
	\end{eqnarray}
For example, the signed series (alternating series) $\sum_{n=1 }^{\infty}\frac{(-1)^{n+1}}{n}$ converges, and it is straightforward to see that the linear density of either set $A$ and $B$ is $\frac{1}{2}$. 

Recently, Gasull and Ma\~{n}osas studied the density of the sets $A$ and $B$ in \eqref{Sets}.

\begin{thm}\label{Gasull}
\textnormal{\cite[Theorem C]{GM}} 
Let $\{c_n\}_n$ be a positive and non-increasing real sequence tending to zero. If the sub-signed series \eqref{Sub-signed} converges, then
	$$
	\lim\limits_{n\to \infty}(A(n)-B(n))c_n=0.
	$$
In addition, if \eqref{Salat-condition} holds, then 
	$$
	\lim_{n\to\infty}(A(n)-B(n))n^{-1}=0.
	$$
\end{thm}

The second statement in Theorem~\ref{Gasull} means that $\operatorname{d}(A)$ exists if and only if $\operatorname{d}(B)$ exists, and in this case, the two densities are equal. 

\medskip
The purpose of this paper is to generalize Theorems~\ref{Aue}--\ref{Gasull} from the linear density $\operatorname{d}$ to an asymptotic $\psi$-density $\operatorname{d}_{\psi}$, the latter of which is defined and  discussed in Section~\ref{density-sec} below. In so doing, our aim is to answer the following research questions.
\begin{itemize}
\item[(Q1)] 
Regarding Theorem~\ref{Aue}, can we find a set $A\subset\N$ satisfying a condition stronger than  $\operatorname{d}(A)=0$ such that the sub-series $\sum_{n=1}^\infty \chi_A(n)c_n$ still diverges? In other words, does the conclusion of Theorem~\ref{Aue} still hold if the set $A$ is replaced with a smaller set in terms of densities? 

\item[(Q2)] 
How are the properties of the sequence $\{c_n\}_n$ in Theorem~\ref{New-Hamm} affected if the assumption $\underline{\operatorname{d}}(A)=0$ is replaced with $\underline{\operatorname{d}}_{\psi}(A)=0$ for a certain unbounded concave function $\psi:(0,\infty)\to (0,\infty)$, given that \eqref{Salat-condition30} and \eqref{convergent-sum0}  still hold?

\item[(Q3)] 
How is the density of $A\subset\N$ in Theorem~\ref{New-Hamm} affected if the assumption that $\{c_n\}_n$ is non-increasing is replaced with either a stronger or a weaker assumption on monotonicity such that \eqref{convergent-sum0} still holds?

\item[(Q4)] 
How is the density of $A\subset\N$ in Theorem~\ref{Salat-thm1} affected if the assumption \eqref{Salat-condition} is replaced with an analogous weighted assumption such that \eqref{convergent-sum0} still holds? 

\item[(Q5)] 
If the assumption that $\{c_n\}_n$ is  non-increasing in Theorem~\ref{Gasull} is replaced with either a stronger or a weaker assumption on monotonicity depending on a certain unbounded function $\psi:(0,\infty)\to (0,\infty)$, and if the sub-signed series \eqref{Sub-signed} converges, then what can be said about the asymptotic $\psi$-densities of the sets $A$ and $B$ defined in \eqref{Sets}?
\end{itemize}


\section{Asymptotic $\psi$-density}\label{density-sec}

We begin by defining a class $\D$ of strictly increasing, differentiable, and unbounded functions $\psi: (0,\infty) \to (0,\infty)$. For convenience, we set
	\begin{equation*}
	A_{\psi }(n):=\sum_{k=1}^n\psi'(k)\chi_A(k)
	\end{equation*}
for any subset $A \subset \N$. A particularly useful property for a function $\psi\in\D$ is that it satisfies 	
	\begin{align}\label{asym}
	\sum_{k=1}^n\psi'(k)\sim \psi(n),\quad n\to\infty.
	\end{align} 
The concepts of $\psi$-density and asymptotic $\psi$-density were introduced in \cite{HL} as follows.

\begin{definition}\label{Def1}
\textnormal{\cite{HL}}
Let $\psi\in\D$. The \textit{lower and upper $\psi$-densities} of a set $A\subset\N$ are defined, respectively, by
\begin{eqnarray*}
	\underline{\operatorname{d}}_\psi(A)
	=\liminf_{n\to\infty}\frac{A_{\psi }(n)}{\sum_{k=1}^n\psi'(k)}	
	\quad\textnormal{and}\quad
\overline{\operatorname{d}}_\psi(A)
	=\limsup_{n\to\infty}\frac{A_{\psi }(n)}{\sum_{k=1}^n\psi'(k)}.
	\end{eqnarray*}
If $\underline{\operatorname{d}}_\psi(A)= \overline{\operatorname{d}}_\psi(A)$, then their common value $\operatorname{d}_\psi(A)$ is called the \textit{$\psi$-density} of $A$. Additionally, if $\psi$ satisfies \eqref{asym}, then the quantities
	\begin{eqnarray*}
	\underline{\operatorname{d}}_\psi(A)
	=\liminf_{n\to\infty}\frac{A_{\psi }(n)}{\psi(n)}	
	\quad\textnormal{and}\quad
\overline{\operatorname{d}}_\psi(A)
	=\limsup_{n\to\infty}\frac{A_{\psi }(n)}{\psi(n)}
	\end{eqnarray*}	
are called the \textit{asymptotic lower and upper $\psi$-densities} of $A$, respectively. Their common value $\operatorname{d}_\psi(A)$, if it exists, is called the \textit{asymptotic $\psi$-density} of $A$.
\end{definition}

Two useful subclasses of $\D$ were introduced in \cite{HL} as follows.

\begin{definition}
\textnormal{\cite{HL}}
Let $\D_1$ denote the class of all concave functions in $\D$, and let $\D_2$ denote the class of all convex functions $\psi$ in $\D$ that satisfy
	\begin{equation}\label{n+1-asymp-n}
	\psi(n+1)\sim \psi(n),\quad n\to\infty.
	\end{equation}
\end{definition}

\begin{lemma}\label{asymptotic-lemma}
\textnormal{\cite{HL}}
If $\psi\in\D_1$, then \eqref{asym} and \eqref{n+1-asymp-n} both hold. If $\psi\in\D_2$, then \eqref{asym} holds.
\end{lemma}

If $A \subset\N$ is any set, then \cite[Corollary~2.8]{HL} shows that 
	\begin{eqnarray}
	0&\leq& \underline{\operatorname{d}}(A) \leq \underline{\operatorname{d}}_\psi(A) \leq \overline{\operatorname{d}}_\psi(A) \leq \overline{\operatorname{d}}(A) \leq 1,\quad\text{if}\ \psi\in\D_1,\label{D1-densities}\\
	0&\leq& \underline{\operatorname{d}}_{\psi}(A) \leq \underline{\operatorname{d}}(A) \leq \overline{\operatorname{d}}(A)\leq \overline{\operatorname{d}}_{\psi}(A) \leq 1,\quad \text{if}\ \psi\in\D_2.\label{D2-densities}
	\end{eqnarray}
Due to these inequalities, the linear density can be thought of as the ``terminal density'' for the asymptotic $\psi$-densities when $\psi \in \D_1$, and as the ``initial density'' for the asymptotic $\psi$-densities when $\psi \in \D_2$.

We close this section with the following result, which generalizes \eqref{LUdens} from $\psi(x)=x$ to $\psi \in \D_1$. 

\begin{theorem}\label{lem1}
\textnormal{\cite{HL}}
Let $A=\{\varphi(n): n \in \mathbb{N}\}$, where $\{\varphi(n)\}_n$ is a strictly increasing sequence of non-negative integers. Then
	\begin{align*}\label{lim1}
		\underline{\operatorname{d}}_{\psi}(A)=\liminf_{n\to\infty}\frac{A_{\psi}(\varphi(n))}{\psi(\varphi(n))} \quad \text{and} \quad \overline{\operatorname{d}}_{\psi}(A)=\limsup_{n\to\infty}\frac{A_{\psi}(\varphi(n))}{\psi(\varphi(n))},
	\end{align*}
for any $\psi \in \D_1$.
\end{theorem}


\section{Generalizations of \v{S}al\'at's theorem:\\ Research question (Q4)}

Theorem~\ref{S1-thm} below generalizes \v{S}al\'at's Theorem~\ref{Salat-thm1} and reduces to it when $\psi(x)=x$. This answers Research question (Q4). The sharpness of the result is discussed in Proposition~\ref{sharpness-prop} below, which will be proved by relying on \v{S}al\'at's findings in \cite[pp.~211--212]{Salat}. Some consequences of the result will also be discussed. 

\begin{theorem}\label{S1-thm}
	Let $\psi\in\D_1\cup\D_2$, and let $\{c_n\}_n$ be a non-negative non-increasing real sequence tending to zero such that 
	\begin{eqnarray}
		\liminf_{n\to\infty}\psi(n)c_n&>&0,\quad \text{if}\ \psi\in\D_1,\label{Salat-condition2A}\\
		\liminf_{n\to\infty}\frac{\psi(n)}{\psi'(n)}c_n&>&0,\quad \text{if}\ \psi\in\D_2.\label{Salat-condition2}
	\end{eqnarray}
	If $A\subset\N$ is an infinite set such that
	\eqref{convergent-sum0} holds, then $\operatorname{d}_{\psi}(A)=0$.
\end{theorem}

\begin{proof}
	For $\psi\in\D_1\cup\D_2$, it suffices to prove that if
	\begin{equation}\label{positive-psi-density}
		\overline{\operatorname{d}}_{\psi}(A)=\limsup_{n\to\infty}\frac{A_\psi(n)}{\psi(n)}>0,
	\end{equation}
	then the series in \eqref{convergent-sum0} diverges. Hence, we suppose that \eqref{positive-psi-density} holds. Then there exists a constant $\delta'>0$ such that
	$$
	A_\psi(n)>\psi(n)\delta'
	$$
	for infinitely many values of $n\in\N$. Also recall from Lemma~\ref{asymptotic-lemma} that  \eqref{asym} holds for $\psi\in\D_1\cup\D_2$.
	
The rest of the proof is divided into two cases.
	
	\medskip
	(1) Suppose that \eqref{Salat-condition2A} holds.
	Then there exists a constant $\delta>0$ and an integer $N_1\in\N$ such that
	$$
	\psi(n)c_n\geq\delta,\quad n\geq N_1.
	$$
	Let $\varepsilon:=\frac{\delta\delta'}{2}>0$, and let $N_2\in\N$ be large enough so that
	$$
	A_\psi(N_2)\leq \sum_{k=1}^{N_2}\psi'(k)\leq 2\psi(N_2).
	$$
	Further, let $N=\max\{N_1,N_2\}$ and choose a large enough constant $n_0\in\N$ satisfying the conditions
	$$
	n_0>N,\quad A_\psi(n_0)>\psi(n_0)\delta',\quad 2\psi(N)c_{n_0}<\frac{\delta\delta'}{2}=\varepsilon.
	$$ 
	This is possible because $c_n\to 0^+$ as $n\to\infty$. Denote the partial sums of the series in \eqref{convergent-sum0} by
	$$
	s_n=\sum_{k=1}^n\chi_A(k)c_k,\quad n\in\N.
	$$
	Then, for the aforementioned values $N$ and $n_0$, we have, by concavity,
	\begin{eqnarray*}
		|s_{n_0}-s_N| &=& s_{n_0}-s_N=\sum_{k=N+1}^{n_0}\chi_A(k)c_k
		\geq c_{n_0}\sum_{k=N+1}^{n_0}\frac{\psi'(k)}{\psi'(k)}\chi_A(k)\\
		&\geq& \frac{c_{n_0}}{\psi'(N+1)}\sum_{k=N+1}^{n_0}\psi'(k)\chi_A(k)
		= \frac{c_{n_0}}{\psi'(N+1)}\big(A_\psi(n_0)-A_\psi(N)\big)\\
		&\geq& \frac{c_{n_0}}{\psi'(N+1)}\big(\psi(n_0)\delta'-2\psi(N)\big)\geq \frac{1}{\psi'(N+1)}\left(\delta\delta'-\frac{\delta\delta'}{2}\right)=\frac{\varepsilon}{\psi'(N+1)}.
	\end{eqnarray*}
	Note that the lower bound is a non-decreasing function of $N$. This means that Cauchy's condition \cite[p.~186]{Apostol} is not fulfilled, and hence the series in \eqref{convergent-sum0} diverges. 
	
	\medskip
	(2) Suppose that \eqref{Salat-condition2} holds.
	Then there exists a constant $\delta>0$ and an integer $N\in\N$ such that
$$
	\frac{\psi(n)}{\psi'(n)}c_n\geq\delta,\quad n\geq N.
	$$
	Let the constants $\varepsilon:=\frac{\delta\delta'}{2}>0$ and  the partial sum $s_n$ be as in Case~(1). Further, choose a large enough constant $n_0\in\N$ satisfying the conditions
	$$
	n_0>N,\quad A_\psi(n_0)>\psi(n_0)\delta',\quad \frac{2\psi(N)}{\psi'(n_0)}c_{n_0}<\frac{\delta\delta'}{2}=\varepsilon.
	$$ 
	This is possible because $c_n\to 0^+$ as $n\to\infty$ and because $\psi'(x)$ is non-decreasing.
	Then, for the aforementioned values $N$ and $n_0$, we have, by convexity,
	\begin{eqnarray*}
		|s_{n_0}-s_N| &=& s_{n_0}-s_N=\sum_{k=N+1}^{n_0}\chi_A(k)c_k
		\geq c_{n_0}\sum_{k=N+1}^{n_0}\frac{\psi'(k)}{\psi'(k)}\chi_A(k)\\
		&\geq& \frac{c_{n_0}}{\psi'(n_0)}\sum_{k=N+1}^{n_0}\psi'(k)\chi_A(k)
		= \frac{c_{n_0}}{\psi'(n_0)}\big(A_\psi(n_0)-A_\psi(N)\big)\\
		&\geq& \frac{c_{n_0}}{\psi'(n_0)}\big(\psi(n_0)\delta'-2\psi(N)\big)\geq \delta\delta'-\frac{\delta\delta'}{2}=\varepsilon.
	\end{eqnarray*}
	This means that Cauchy's condition is not fulfilled, and hence the series in \eqref{convergent-sum0} diverges. 	 
\end{proof}

We proceed to discuss the sharpness of Theorem~\ref{S1-thm}.
To this end, we need the obvious fact that the series $\sum_{n=1}^\infty \frac{\psi'(n)}{\psi(n)}$ diverges for $\psi\in\D_1\cup\D_2$.
If 
	\begin{equation}\label{Salat-condition2X}
	\liminf_{n\to\infty}\frac{\psi(n)}{\psi'(n)}c_n>0
	\end{equation}
holds for some $\psi\in\D_2$, then there exist constants $\gamma>0$ and $N\in\N$ such that 
	$$
	c_n\geq \gamma\frac{\psi'(n)}{\psi(n)},
	\quad n\geq N.
	$$
Consequently, \eqref{Salat-condition2X} implies $\sum_{n=1}^\infty c_n=\infty$.
Suppose then that \eqref{Salat-condition2X} is replaced with the weaker condition $\sum_{n=1}^\infty c_n=\infty$, and let $\delta\in \left( 0,1\right] $ and $A\subset\N$ be as in Proposition~\ref{Salat-example}. Then $\overline{\operatorname{d}}_\psi(A)\geq\overline{\operatorname{d}}(A)= \delta$ by \eqref{D2-densities}, which means that the conclusion of Theorem~\ref{S1-thm} fails. Similarly, if
	\begin{equation}\label{Salat-condition2Y}
	\liminf_{n\to\infty}\psi(n)c_n>0
	\end{equation}
holds for some $\psi\in\D_1$, then
	\begin{equation}\label{XY}
	\sum_{n=1}^\infty \psi'(n)c_n=\infty.
	\end{equation}
We prove that if \eqref{Salat-condition2Y} is replaced with a weaker condition \eqref{XY} for $\psi\in\D_1$, then the conclusion of Theorem~\ref{S1-thm} fails. 

\begin{proposition}\label{sharpness-prop}
For every $\delta\in (0,1)$ there exist $\psi\in\D_1$, a set $A\subset\N$ and a sequence $\{c_n\}_n$ of positive real numbers with the following properties:
	$$
	\sum_{n=1}^\infty \psi'(n)c_n=\infty,\quad \sum_{n=1}^{\infty}\chi_A(n)c_n<\infty,\quad \overline{\operatorname{d}}_\psi(A)\geq \delta.
	$$
\end{proposition}

 The proof of Proposition~\ref{sharpness-prop} is given in Appendix~\ref{App-A} along with the alternative proof of Proposition~\ref{Salat-example} since these two proofs are parallel. 
 We proceed to present a result in the spirit of Theorem~\ref{S1-thm}, which is valid under a milder  assumption  $\psi\in\D$.

\begin{theorem}\label{Th2}
		Let $\psi\in \D$, and let $\{c_n\}_n$ be a non-negative real sequence such that $\left\{\frac{c_n}{\psi'(n)}\right\}_n$ is non-increasing and tending to zero. 	If $A\subset\N$ is an infinite set such that
		\eqref{convergent-sum0} holds, then
			\begin{eqnarray}\label{NC1}
			 \frac{A_{\psi}(n)}{\psi'(n)}c_n\to 0,\quad n\to \infty.
		\end{eqnarray}
In addition, if $\psi$ satisfies \eqref{asym} and
		\begin{eqnarray}\label{Suf1}
			\liminf_{n\to\infty}\frac{\psi(n)}{\psi'(n)}c_n>0,
		\end{eqnarray}
		then $\operatorname{d}_{\psi}(A)=0$. 
\end{theorem}

The proof of Theorem~\ref{Th2} requires the following elementary lemma.

\begin{lemma}\label{lem3}
	Let $\{a_n\}_n$ and $\{b_n\}_n$ be two real sequences such that $\lim\limits_{n\to \infty}a_nb_n=0$. If $\liminf\limits_{n\to \infty} a_n > 0$, then $\lim\limits_{n\to \infty}b_n = 0$.
\end{lemma}

\begin{proof}
Given $\lim\limits_{n\to \infty} a_nb_n = 0$, let $\varepsilon > 0$ be arbitrary. By the definition of limit, there exists an integer $N_1$ such that 
	$$
	|a_n b_n|< \varepsilon, \quad \text{for all}\;n\geq N_1.
	$$
Now, since $\liminf\limits_{n\to \infty} a_n > 0$, there exists a real number $\delta >0$ and an integer $N_2$ such that 
	$$
	|a_n|\geq \delta, \quad \text{for all}\; n\geq N_2.
	$$
Thus, for any $n \geq \max \{N_1, N_2\}$, we have
	$$
	|b_n|=\frac{|a_nb_n|}{|a_n|}\leq \frac{\varepsilon}{\delta}.
	$$
This inequality holds for any $\varepsilon > 0$, which implies that $\lim\limits_{n\to \infty}b_n = 0$.
\end{proof}
	\begin{proof}[Proof of Theorem~\ref{Th2}]
	 By Abel's partial summation formula \cite[p.~194]{Apostol}, we obtain 
		\begin{eqnarray}\label{Id1}
			\sum_{k=1}^n\chi_A(k)c_k=c_n\frac{A_{\psi}(n)}{\psi'(n)}+\sum_{k=1}^{n-1}v_k,
		\end{eqnarray}
		where $\{v_n\}_n$ is non-negative real sequence defined by
		$$
		v_n=A_{\psi}(n)\left( \frac{c_n}{\psi'(n)}-	\frac{c_{n+1}}{\psi'(n+1)}\right).
		$$
From \eqref{Id1} and \eqref{convergent-sum0}, we deduce that the series $\sum_{n=1}^{\infty} v_n$ converges. Since $\left\{\frac{c_n}{\psi'(n)}\right\}_n$ tends to zero, we have
		\begin{eqnarray*}
			\frac{c_n}{\psi'(n)}=\sum_{k\geq n} \frac{v_k}{A_{\psi}(k)} \leq \frac{1}{A_{\psi}(n)}\sum_{k\geq n} v_k.
		\end{eqnarray*}
	This, combined with the convergence of $\sum_{n=1}^{\infty} v_n$, implies that the remainder $\sum_{k\geq n} v_k$ converges to $0$ as $n\to \infty$, and the assertion \eqref{NC1} follows immediately. 
	
Finally, given that \eqref{Suf1} holds, we may use Lemma~\ref{lem3}, with
	$$
	a_n=\frac{\psi(n)}{\psi'(n)}c_n \quad \text{and}\quad b_n=\frac{A_{\psi}(n)}{\psi(n)},
	$$
to conclude that $\operatorname{d}_{\psi}(A) = 0$. 
\end{proof}

\begin{remark}
	(1)	The assertion \eqref{NC1} is necessary for \eqref{convergent-sum0} but not sufficient. Indeed, if $\Pi\subset\N$ denotes the set of prime numbers, then the series $\sum\limits_{p\in \Pi}\frac{1}{p}$ diverges. Nevertheless, \eqref{NC1} is satisfied for $\psi(x)=x$ since we have $c_n=\frac{1}{n}$, and we obtain
	$$
	\operatorname{d}(\Pi)=\lim_{n\to \infty}\frac{\Pi(n)}{n}=\lim_{n\to \infty} \frac{1}{\log n}= 0.
	$$
Here, $\Pi(n)$ denotes the prime-counting function, known from the prime number theorem to have an asymptotic behavior given by $\Pi(n)\sim \frac{n}{\log n}$.

(2) Assertion (a) in \cite[Theorem A]{GM} follows as a special case of Theorem~\ref{Th2} by choosing $\psi(x)=x$. However, our method of proof is different from that in \cite{GM}.
\end{remark}

\begin{example}
Recall from \cite[Remark~4.4]{HL} that the maximal growth rate for a function $\psi\in\D_2$ is $\log \psi(x)=o(x)$. There are functions $\psi\in\D_2$ of near maximal growth illustrating the validity of \eqref{Suf1} as a criterion for obtaining $\operatorname{d}_{\psi}(A)=0$ when \eqref{convergent-sum0} is true.  For example, let $A\subset\N$ be an infinite set such that 
	$$
	\sum\limits_{n\in A}\frac{1}{n^{\alpha}} <\infty, \quad 0<\alpha <1.
	$$
Choose $\psi(x)=\exp(x^{\delta})$ for $0<\delta\leq 1-\alpha$, so that $\psi\in\D_2$. Then \eqref{Suf1} is satisfied, and consequently, $\operatorname{d}_{\psi}(A)=0$.
\end{example}

The classical Olivier's theorem \cite[p.~34]{O}  asserts that for a non-increasing positive real sequence $\{c_n\}_n$, if $\sum_{n=1}^{\infty}c_n$ converges, then $nc_n\to 0$ as $n\to \infty$. Olivier's theorem is also known as  Abel's theorem or Pringsheim's theorem \cite[p.~350]{Hardy}.
Theorem~\ref{Th2} can be seen as an extension of Olivier's theorem. Indeed, by choosing $A=\N$, we see that \eqref{convergent-sum0} is equivalent to the convergence of the series $\sum_{n=1}^{\infty}c_n$, which consequently implies $c_n\to 0$ as $n\to\infty$. Furthermore, with $\psi(x)=x$, \eqref{NC1} can be expressed as
	$$
	A(n)c_n=nc_n\to 0, \quad n\to \infty.
	$$

The following result is another extension of Olivier's theorem.

\begin{corollary}\label{Cor1}
Suppose that $\psi\in\D$ satisfies \eqref{asym}, and let $\{c_n\}_n$ be a non-negative real sequence such that $\left\{\frac{c_n}{\psi'(n)}\right\}_n$ is non-increasing. If $\sum_{n=1}^{\infty}c_n$ converges, then
	\begin{equation}\label{limit0}
		\lim_{n\to \infty}\frac{\psi(n)}{\psi'(n)}c_n=0.
	\end{equation}
In particular, if $\left\{\frac{nc_n}{(\log n)^{\alpha-1}}\right\}_{n\geq 2}$ is non-increasing for some $\alpha\geq 1$, then
	\begin{eqnarray}\label{limit}
	\lim _{n\to \infty}n(\log n)c_n=0.
	\end{eqnarray}
\end{corollary}

\begin{proof}
To begin with, we show that the sequence $\left\{\frac{c_n}{\psi'(n)}\right\}_n$ converges to zero. Assuming the contrary, and keeping in mind that $\left\{\frac{c_n}{\psi'(n)}\right\}_n$ is non-increasing, we infer the existence of $\delta > 0$ such that
	$$
	c_n\geq \delta \psi'(n).
	$$
Consequently, we have
	$$
	\sum_{k=1}^{n} c_k\geq \delta \sum_{k=1}^{n}\psi'(k)\sim \delta \psi(n),
	$$
which contradicts with the convergence of  $\sum_{n=1}^{\infty}c_n$. 

To establish \eqref{limit0}, it suffices to set $A=\N$ in \eqref{NC1} and to use \eqref{asym}.
Finally, to prove \eqref{limit}, we make the specific choice $\psi(x)=(\log x)^\alpha$, and then apply \eqref{limit0}.
\end{proof}

The case $\alpha=1$ in Corollary~\ref{Cor1}, i.e., $\{na_n\}_n$ is non-increasing, is observed in \cite[Remark~1]{NP} by using an estimate on harmonic numbers.


\section{Generalization of Hamming's theorem:\\ Research questions (Q2) and (Q3) }

We begin this section by generalizing the first implication of Theorem~\ref{New-Hamm} from $\psi(x) = x$ to $\psi \in \mathcal{D}_1$. For the sake of brevity, for any $\psi \in \mathcal{D}_1$, we denote by $\mathcal{I}_{\psi}$ the class of non-negative real sequences $\{c_n\}_n$ satisfying \eqref{Salat-condition30} for which $\left\{\frac{c_n}{\psi'(n)}\right\}_n$ is non-increasing. Note that if $\psi(x) = x$, then $\mathcal{I}_\psi$ reduces to $\mathcal{I}$.

\begin{theorem}\label{psi-Hamming}
Let $\psi\in \D_1$, and let $A \subset \N$ be an infinite set. If $\underline{\operatorname{d}}_\psi(A)=0$, then there exists a positive sequence $\{c_n\}_n \in \mathcal{I}_{\psi}$ such that \eqref{convergent-sum0} holds.
\end{theorem}

Before we proceed to the proof of Theorem~\ref{psi-Hamming}, we can, without loss of generality, express the set $A$ in Theorem~\ref{psi-Hamming} as $A = \{\varphi(n) : n \in \mathbb{N}\}$, where $\{\varphi(n)\}_n$ is an arbitrary strictly increasing sequence of non-negative integers. If $\psi \in \mathcal{D}_1$, then Theorem~\ref{lem1} gives
	$$
	\underline{\operatorname{d}}_\psi(A) =
\liminf_{n \to \infty} \frac{A_\psi(\varphi(n))}{\psi(\varphi(n))},
	$$
where
	$$
	A_{\psi}(\varphi(n)) = \sum_{k=1}^{\varphi(n)}\psi'(k)\chi_A(k) = \sum_{k=1}^n \psi'(\varphi(k)).
	$$
In particular, \(\underline{\operatorname{d}}_\psi(A) = 0\) if and only if
\begin{equation}\label{LS3}
	\limsup_{n \to \infty} \frac{\psi(\varphi(n))}{A_\psi(\varphi(n))} = \infty.
\end{equation}

Using the above notation, we present the following two elementary lemmas, which are essential for the proof.

\begin{lemma}\label{lem4}
Under the assumptions of Theorem~\ref{psi-Hamming}, there exists a strictly increasing sequence $\{n_k\}_k$ of positive integers tending to infinity and satisfying
	\begin{eqnarray}\label{Hyp1}
		\frac{\psi(\varphi(n_{k+1}))-\psi(\varphi(n_k))}{A_{\psi}(\varphi(n_{k+1}))-A_{\psi}(\varphi(n_k))}\geq 2^k
	\end{eqnarray}
	and 
	\begin{eqnarray}\label{Hyp2}
		A_{\psi}(\varphi(n_{k+1}))-A_{\psi}(\varphi(n_k))\geq A_{\psi}(\varphi(n_{k}))-A_{\psi}(\varphi(n_{k-1})).
	\end{eqnarray}
\end{lemma}

\begin{proof}
	Assume first that there exists a constant $C>0$ such that
	\begin{equation}\label{C-inequality}
		\frac{\psi(\varphi(n+1))-\psi(\varphi(n))}{A_{\psi}(\varphi(n+1))-A_{\psi}(\varphi(n))}\leq C,\quad n\in\N.
	\end{equation}
	Note that $\psi(\varphi(n+1))-\psi(\varphi(n))\geq 0$ and that
	$$
	A_{\psi}(\varphi(n+1))-A_{\psi}(\varphi(n))=\psi'(\varphi(n+1))>0
	$$
	because $\psi\in\D_1$. Applying inequality \eqref{C-inequality} repeatedly, we obtain
	\begin{eqnarray*}
		\psi(\varphi(n+1))-\psi(\varphi(n)) &\leq& C\big(A_{\psi}(\varphi(n+1))-A_{\psi}(\varphi(n))\big)\\
		\psi(\varphi(n))-\psi(\varphi(n-1)) &\leq& C\big(A_{\psi}(\varphi(n))-A_{\psi}(\varphi(n-1))\big)\\
		&\vdots&\\
		\psi(\varphi(2))-\psi(\varphi(1)) &\leq& C\big(A_{\psi}(\varphi(2))-A_{\psi}(\varphi(1))\big).
	\end{eqnarray*} 
	By summing these inequalities, it follows that
	$$
	\psi(\varphi(n+1))-\psi(\varphi(1)) \leq C\big(A_{\psi}(\varphi(n+1))-A_{\psi}(\varphi(1))\big).
	$$
	But this violates \eqref{LS3} and consequently the assumption $\underline{\operatorname{d}}_\psi(A)=0$. Hence we conclude that there exists a strictly increasing sequence $\{n_k\}_k$ of positive integers tending to infinity such that the quotient in \eqref{Hyp1} is unbounded. By choosing a suitable subsequence of $\{n_k\}_k$, denoted again by $\{n_k\}_k$, we find that \eqref{Hyp1} holds. On the other hand, keeping in mind that $A=\{\varphi(n):n\in\N\}$, the condition \eqref{Hyp2} can be written equivalently as
	$$
	\sum_{m=n_k+1}^{n_{k+1}}\psi'(\varphi(m))\geq \sum_{m=n_{k-1}+1}^{n_{k}}\psi'(\varphi(m)).
	$$
	But since $\psi'>0$, this condition can be assumed by choosing a subsequence of $\{n_k\}_k$, denoted again by $\{n_k\}_k$. This completes the proof. 
\end{proof}

\begin{lemma}\label{lem2}
	Let $\psi\in \D_1$, then for any integers $n,m$ with $1\leq n<m$, we have
	$$
	\sum_{j=n+1}^{m}\psi'(j)
	\leq \psi(m)-\psi(n) \leq 	\sum_{j=n+1}^{m}\psi'(j) +\psi'(n)-\psi'(m).
	$$
\end{lemma}

\begin{proof}
	Since $\psi^{\prime}$ is a non-increasing and a non-negative function, 
	\begin{eqnarray*}
		\psi(m)-\psi(n)&=&\int_{n}^{m}\psi'(t)\, dt
		=\sum_{j=n}^{m-1}\int_j^{j+1}\psi'(t)\, dt
		\leq \sum_{j=n}^{m-1}\psi'(j).
	\end{eqnarray*}
	This implies the upper bound. The lower bound is proved similarly.
\end{proof}

Now we are ready to prove Theorem~\ref{psi-Hamming}.
\begin{proof}[Proof of Theorem~\ref{psi-Hamming}]
	Let $\{n_k\}_k$ be a strictly increasing sequence of positive integers tending to infinity such that \eqref{Hyp1} and \eqref{Hyp2} hold. Set $n_0=0$. Without loss of generality, we may suppose that $\varphi(n_0)=\varphi(0)=0$.
	
	For each $k\in\N\cup\{0\}$, choose constants $B(m)$ depending on $m$ such that
	\begin{equation}\label{B-constants2}
		2\geq B(\varphi(n_k)+1)\geq B(\varphi(n_k)+2)\geq \cdots \geq B(\varphi(n_{k+1}))\geq 1.
	\end{equation}
	There are $\varphi(n_{k+1})-\varphi(n_k)$ such constants for each $k$. 
	Then define a sequence $\{c_m\}_m$ of positive real numbers by
	$$
	c_m=\frac{B(m)\psi'(m)}{2^k(A_{\psi}(\varphi(n_{k+1}))-A_{\psi}(\varphi(n_{k})))},\quad
	\varphi(n_k)+1\leq m\leq \varphi(n_{k+1}),\ k\in\N\cup\{0\}.
	$$
	From \eqref{B-constants2}, it is clear that 
	$$
	\frac{c_{\varphi(n_k)+1}}{\psi'(\varphi(n_k)+1)}\geq \frac{c_{\varphi(n_k)+2}}{\psi'(\varphi(n_k)+2)}\geq\cdots\geq \frac{c_{\varphi(n_{k+1})}}{\psi'(\varphi(n_{k+1}))}.
	$$ 
	Moreover, \eqref{Hyp2} yields
	\begin{eqnarray*}
		\frac{B(\varphi(n_{k+1})+1)}{B(\varphi(n_{k+1}))}&\leq& 2
		\leq 2\cdot\frac{A_{\psi}(\varphi(n_{k+2}))-A_{\psi}(\varphi(n_{k+1}))}{A_{\psi}(\varphi(n_{k+1}))-A_{\psi}(\varphi(n_{k}))}\\
		&=&\frac{2^{k+1}(A_{\psi}(\varphi(n_{k+2}))-A_{\psi}(\varphi(n_{k+1})))}{2^k(A_{\psi}(\varphi(n_{k+1}))-A_{\psi}(\varphi(n_{k})))},
	\end{eqnarray*}
	and thus 
	$$
	\frac{c_{\varphi(n_{k+1})}}{\psi'(\varphi(n_{k+1}))}\geq \frac{c_{\varphi(n_{k+1})+1}}{\psi'(\varphi(n_{k+1})+1)}
	$$ 
	follows by the definition of the numbers $\frac{c_m}{\psi'(m)}$. Therefore, the sequence $\left\lbrace \frac{c_m}{\psi'(m)}\right\rbrace _m$ is non-increasing, and clearly tends to zero as $m\to\infty$.
	
	It remains to prove that $\sum_{n=1}^\infty c_n$ diverges and \eqref{subseries} converges. By using Lemma~\ref{lem2} together with \eqref{Hyp1} and \eqref{Hyp2}, we obtain
	\begin{eqnarray*}
		\sum_{m=\varphi(n_k)+1}^{\varphi(n_{k+1})}c_m &=&
		\sum_{m=\varphi(n_k)+1}^{\varphi(n_{k+1})}\frac{\psi'(m)B(m)}{2^k(A_{\psi}(\varphi(n_{k+1}))-A_{\psi}(\varphi(n_{k})))}\\
		&\geq& \frac{\sum_{m=\varphi(n_k)+1}^{\varphi(n_{k+1})}\psi'(m)}{2^k(A_{\psi}(\varphi(n_{k+1}))-A_{\psi}(\varphi(n_{k})))}\\
		&\geq&\frac{\psi(\varphi(n_{k+1}))-\psi(\varphi(n_k))}{2^k(A_{\psi}(\varphi(n_{k+1}))-A_{\psi}(\varphi(n_{k})))} -\frac{\psi'(\varphi(n_{k}))-\psi'(\varphi(n_{k+1}))}{2^k(A_{\psi}(\varphi(n_{k+1}))-A_{\psi}(\varphi(n_{k})))}\\
		&\geq & 1-\frac{\psi'(\varphi(n_{k}))-\psi'(\varphi(n_{k+1}))}{2^k(A_{\psi}(\varphi(n_{k+1}))-A_{\psi}(\varphi(n_{k})))}\geq 1-\frac{M}{2^k}
	\end{eqnarray*}
	for some positive constant $M$. Hence,
	the series $\sum_{n=1}^\infty c_n$ diverges. 
	On the other hand, 
	\begin{eqnarray*}
		\sum_{m=n_k+1}^{n_{k+1}}c_{\varphi(m)}&=&\sum_{m=n_k+1}^{n_{k+1}}\frac{B(\varphi(m))\psi'(\varphi(m))}{2^k(A_{\psi}(\varphi(n_{k+1}))-A_{\psi}(\varphi(n_{k})))}
		\\
		&\leq&\frac{1}{2^{k-1}(A_{\psi}(\varphi(n_{k+1}))-A_{\psi}(\varphi(n_{k})))}\sum_{m=n_k+1}^{n_{k+1}}\psi'(\varphi(m))=\frac{1}{2^{k-1}},
	\end{eqnarray*}
	and hence $\sum_{n=1}^\infty c_{\varphi(n)}$ converges. This completes the proof.
\end{proof}

Next, we proceed to generalize the second implication of Theorem~\ref{New-Hamm} (proved by \v{S}al\'at in \cite[Theorem~2]{Salat}) from $\psi(x)=x$ to $\psi\in\D$. From the densities point of view, the assumption ``eventually'' in \cite[Theorem~2]{Salat} can be dropped without any loss of generality.

\begin{theorem}\label{Salat2-thm}
Let $\psi\in\D$, and let $\{c_n\}_n$ be a non-negative real sequence satisfying \eqref{Salat-condition30}. If $A\subset\N$ is an infinite set satisfying \eqref{convergent-sum0}, then the following assertions~hold.
	\begin{itemize}
		\item [\textnormal{(a)}] If $\left\{\frac{c_n}{\psi'(n)}\right\}_n$ is non-increasing, then $\underline{\operatorname{d}}_\psi(A)=0$.
		
		\item [\textnormal{(b)}] If $\left\{\frac{c_n}{\psi'(n)}\right\}_n$ is non-decreasing, then $\operatorname{d}_\psi(A)=0$.
	\end{itemize} 
\end{theorem}

The proof of Theorem~\ref{Salat2-thm} is based on the following  Rajagopal's result. 

\begin{thm}\label{L1}
	\textnormal{\cite[Theorem~3]{R}}
	Let $\{s_n\}_n$ be a real sequence, and let $\{a_n\}_n$ and $\{b_n\}_n$ be positive real sequences such that
	$$
	A_n:=\sum_{k=1}^na_k\to \infty,\quad B_n:=\sum_{k=1}^nb_k\to \infty,\quad n\to \infty.
	$$
	Denote
	$$
	\sigma(s_n,a_n)=\frac{\sum_{k=1}^na_ks_k}{A_n} 
	\quad\text{and}\quad
	\sigma(s_n,b_n)=\frac{\sum_{k=1}^nb_ks_k}{B_n}.
	$$
	If $\{a_n/b_n\}_n$ is non-increasing, then
	$$
	\liminf_{n\to\infty}\sigma(s_n,b_n) \leq \liminf_{n\to\infty}\sigma(s_n,a_n)\leq \limsup_{n\to\infty}\sigma(s_n,a_n)\leq\limsup_{n\to\infty}\sigma(s_n,b_n).
	$$
\end{thm}

\begin{proof}[Proof of Theorem~\ref{Salat2-thm}]
	(a) Suppose that $\left\{\frac{c_n}{\psi'(n)}\right\}_n$ is non-increasing. Choosing $a_n = c_n$, $b_n = \psi'(n)$ and $s_n = \chi_A(n)$ in Theorem~\ref{L1}, we get
	\begin{eqnarray*}
		\underline{\operatorname{d}}_\psi(A)&=&\liminf_{n\to\infty}\frac{\sum_{k=1}^n\chi_A(k)\psi'(k)}{\sum_{k=1}^n\psi'(k)} \leq \liminf_{n\to\infty}\frac{\sum_{k=1}^n\chi_A(k)c_k}{\sum_{k=1}^nc_k} \\
		&\leq& \limsup_{n\to\infty}\frac{\sum_{k=1}^n\chi_A(k)c_k}{\sum_{k=1}^nc_k} \leq \limsup_{n\to\infty}\frac{\sum_{k=1}^n\chi_A(k)\psi'(k)}{\sum_{k=1}^n\psi'(k)}= \overline{\operatorname{d}}_\psi(A).
	\end{eqnarray*}
	In particular,
	$$
	0 \leq \underline{\operatorname{d}}_\psi(A) \leq \liminf_{n\to\infty}\frac{\sum_{k=1}^n\chi_A(k)c_k}{\sum_{k=1}^nc_k} \leq \limsup_{n\to\infty}\frac{\sum_{k=1}^n\chi_A(k)c_k}{\sum_{k=1}^nc_k}.
	$$
	From \eqref{convergent-sum0} and \eqref{Salat-condition30}, we know that 
	\begin{eqnarray}\label{lim}
		\lim_{n\to\infty}\frac{\sum_{k=1}^n\chi_A(k)c_k}{\sum_{k=1}^nc_k}=0.
	\end{eqnarray}
	Hence the assertion $\underline{\operatorname{d}}_\psi(A)=0$ follows.
	
	(b) Suppose then that $\left\{\frac{c_n}{\psi'(n)}\right\}_n$ is non-decreasing, in which case $\left\{\frac{\psi'(n)}{c_n}\right\}_n$ is non-increasing. 
	Using Theorem~\ref{L1} again, we obtain
	$$
	0\leq  \liminf_{n\to\infty}\frac{\sum_{k=1}^n\chi_A(k)c_k}{\sum_{k=1}^nc_k}\leq \underline{\operatorname{d}}_\psi(A)  \leq  \overline{\operatorname{d}}_\psi(A) \leq  \limsup_{n\to\infty}\frac{\sum_{k=1}^n\chi_A(k)c_k}{\sum_{k=1}^nc_k}.
	$$
	Then, by making use of \eqref{lim} again, the assertion $\operatorname{d}_\psi(A)=0$ follows. 
\end{proof}

\begin{remark}
	(a) Let $\psi\in\D_1$, and let $\{c_n\}_n$ be a non-negative real sequence such that  $\left\{\frac{c_n}{\psi'(n)}\right\}_n$ is non-increasing. Since $\psi'$ is non-increasing, the sequence $\{c_n\}_n$ is non-increasing.  Hence, comparing Theorem~\ref{Salat2-thm}(a) and the converse implication of Theorem~\ref{New-Hamm}, we see that a stronger assumption yields a stronger result in the~sense~of~\eqref{D1-densities}. 
	
	On the other hand, if $\{c_n\}_n$ is non-increasing and $\psi\in\D_2$, then $\left\{\frac{c_n}{\psi'(n)}\right\}_n$ is non-increasing. Comparing Theorem~\ref{Salat2-thm}(a) and and the converse implication of Theorem~\ref{New-Hamm} in this case, we see that a weaker assumption yields a weaker result in the sense of \eqref{D2-densities}.
	
	(b) Let $\psi\in\D_2$, and suppose that $\{c_n\}_n$ is a non-negative real sequence such that  $\left\{\frac{c_n}{\psi'(n)}\right\}_n$ is non-decreasing. It follows that $\{c_n\}_n$ is non-decreasing. But this means that $\sum_{n=1}^\infty \chi_A(n)c_n=\infty$ for any infinite subset $A\subset\N$. Consequently, Theorem~\ref{Salat2-thm}(b) is not applicable in the case of $\psi\in\D_2$.
	
	On the other hand, if $\psi\in\D_1$, then $\{c_n\}_n$ being non-increasing (as in \cite[Theorem~2]{Salat}) has no correlation with $\left\{\frac{c_n}{\psi'(n)}\right\}_n$ being non-decreasing (as in Theorem~\ref{Salat2-thm}(b)).
\end{remark}

An alternative proof of Theorem~\ref{Salat2-thm} when $\psi\in\D$ is log-concave and satisfies \eqref{asym} is given in Appendix~\ref{App-B} by relying on Abel's partial summation formula instead of Rajagopal's result. 
By definition, if $\psi\in\D$ is log-concave, then $(\log \psi(x))'=\psi'(x)/\psi(x)$ is non-increasing, and hence bounded from above. Therefore, the maximal growth-rate for a log-concave function $\psi$ is $\log \psi(x)=O(x)$. Every function $\psi\in\D_1$ is clearly log-concave, while functions $\psi\in\D_2$ may or may not be log-concave \cite{HL}.

By combining Theorems~\ref{psi-Hamming} and \ref{Salat2-thm}(a), we derive the following corollary, which provides a general necessary and sufficient conditions for the divergence of the sub-series \eqref{subseries}.

\begin{corollary}
	Let $\psi\in \D_1$, and let $A\subset \N$ be an infinite set. The following two assertions are equivalent:
	\begin{itemize}
		\item[\textnormal{(a)}] $\overline{	\operatorname{d}}_{\psi}(A)>0$;
		\item[\textnormal{(b)}] \eqref{subseries} diverges for any $\{c_n\}_n \in \mathcal{I}_{\psi}$.
	\end{itemize}
\end{corollary}


\section{Generalization of Auerbach's theorem:\\ Research question (Q1)}

Motivated by Theorem~\ref{Aue}, one may ask how small the set \(A\) can be in terms of the asymptotic $\psi$-density to ensure the divergence of $\sum_{n=1}^\infty \chi_A(n)c_n$? In other words, what is the maximal growth rate for $\psi \in \mathcal{D}_2$ that satisfies $\operatorname{d}_\psi(A)=0$ in Theorem~\ref{Aue}? To this end, the reader is invited to recall the inequalities in \eqref{D2-densities}.

In the following, we present a \(\psi\)-version of Theorem~\ref{Aue} for the case where \(\psi \in \mathcal{D}_2\) under a certain additional condition. The proof of Theorem~\ref{Th1} follows in large extent the one in  \cite[Lemma, p.~116]{NS}.

\begin{theorem}\label{Th1}
Let $\{c_n\}_n$ be a non-negative real sequence such that \eqref{Salat-condition30} holds. Then there exists a set $A \subset \N$ such that $\sum_{n=1}^\infty \chi_A(n)c_n=\infty$ and
	\begin{eqnarray}\label{Cond}
	\lim\limits_{n\to \infty}\frac{A_\psi(n)}	{n\psi'(n)}=0 
	\end{eqnarray}
holds for any $\psi \in \D_2$.
In addition, if $\psi\in\D_2$ satisfies
	\begin{equation}\label{Cond2}
	\liminf_{n\to \infty}\frac{\psi(n)}{n\psi'(n)}>0,
	\end{equation}
then $\operatorname{d}_{\psi}(A)=0$.
\end{theorem}

\begin{remark}
(a) The assertion \eqref{Cond} is weaker than the assertion $\operatorname{d}_{\psi}(A)=0$ because for any $\psi\in \D_2$, we have 
	$$
	\frac{A_{\psi}(n)}{\psi(n)} \geq (1+o(1))\frac{A_\psi(n)}{n\psi'(n)}
	$$
for all $n$ large enough. This follows from the asymptotic inequality
	$$
	\psi(n)\sim \sum_{k=1}^{n}\psi'(k)\leq n\psi'(n), \quad n \to \infty.
	$$
	
(b) If $\psi(x)=x$, then \eqref{Cond2} is trivially valid, and the second assertion in Theorem~\ref{Th1} coincides with Theorem~\ref{Aue}. In general, Theorem~\ref{Th1} extends Theorem~\ref{Aue} by means of \eqref{D2-densities}, answering Research question (Q1).
\end{remark}

\medskip
\noindent
\emph{Proof of Theorem~\ref{Th1}}.
Let $\psi\in\D_2$. We shall show that if the series $\sum_{n=1}^\infty c_n$ is divergent, then we can construct a set $A$ on which the series also diverges and satisfies \eqref{Cond}. To begin with, we establish a binary tree with countably infinite height, where the nodes represent infinite subsets of $\N$. We initiate this tree with the root $\N$. Subsequently, for any node $N$ within the existing tree, we define two successors, denoted as $N'$ and $N''$, as follows: Let $N=\left\lbrace n_1, n_2, \ldots\right\rbrace$, where $n_i < n_{i+1}$. Then, we set 
	$$
	N'=\left\lbrace n_1, n_3, n_5, \ldots\right\rbrace \quad \text{and} \quad N''=\left\lbrace n_2, n_4, n_6, \ldots\right\rbrace.
	$$
Let us now select a cofinal branch in the tree. Let $N^{\prime}=\{1,3, \ldots\}$ and $N^{\prime \prime}=\{2,4, \ldots\}$ be the successors of $\N$. Then either $\sum_{n=1}^\infty \chi_{N^{\prime}}(n)c_n$ or $\sum_{n=1}^\infty \chi_{N^{\prime \prime}}(n) c_n$ is divergent. Accordingly, let $N_1=N^{\prime}$ or $N_1=N^{\prime \prime}$. Choose $n_0=0$ and $n_1$ such that 
	$$
	\sum_{i=n_0+1 }^{n_1} \chi_{N_1}(i)c_i>1.
	$$ 
Now let $N^{'}, N^{''}$ be the successors of $N_1$. Then either $\sum_{n=1}^\infty \chi_{N^{'}} c_n$ or $\sum_{n=1}^\infty\chi_{N^{\prime \prime}}(n) c_n$ is divergent. Accordingly choose $N_2=N^{\prime}$ or $N_2=N^{\prime \prime}$. Then choose $n_2>2 n_1$ such that 
	$$
	\sum_{i=n_1+1}^{n_2} \chi_{N_2}(i)c_i>2.
	$$
Suppose $N_1, N_2, N_3, \ldots$ have been chosen in this way. Accordingly, we define
	$$
	A=\bigcup_{i \geq 1} A_i, \quad \text{where}\quad A_i=N_i\cap \left[n_{i-1}+1,n_i \right]. 
	$$
Clearly $\sum_{n= 1}^{\infty} \chi_A(n) c_n=\infty.$
It remains to prove \eqref{Cond}. For any $n \in \N$, there exists an integer $r=r(n)$ such that $n_{r-1}<n \leqslant n_{r}$. By the convexity of $\psi$ and the fact that $n_j>2n_{j-1}$, we have
	\begin{eqnarray*}
	A_{\psi}(n)&=& \sum_{j=1}^{r-1}\sum_{i=n_{j-1}+1}^{n_j}\chi_{N_j}(i)\psi'(i)+ \sum_{i=n_{r-1}+1}^{n}\chi_{N_r}(i)\psi'(i)\\
	&\leq&\left( \sum_{j=1}^{r-1}\sum_{i=n_{j-1}+1}^{n_j}\chi_{N_j}(i)+ \sum_{i=n_{r-1}+1}^{n}\chi_{N_r}(i)\right) \psi'(n) \\
	&=&\left( \sum_{j=1}^{r-1}\frac{n_j-n_{j-1}}{2^j}+ \frac{n-n_r}{2^r}\right)  \psi'(n) \\
	&\leq&\frac{1}{2^{r-1}}\left( r-\frac{1}{2}\right) n\psi'(n),
	\end{eqnarray*}
and \eqref{Cond} follows immediately. The second assertion $\operatorname{d}_{\psi}(A)=0$ is trivial.
\hfill$\Box$

\medskip
The rest of this section is devoted to studying the assumption \eqref{Cond2}. These discussions are not directly related to the main research questions, so they can be skipped. However, we do address the question on the maximal growth rate of $\psi \in \mathcal{D}_2$, raised at the beginning of this section. The answer lies in the concept of order of growth.

\begin{lemma}
If $\psi\in\D$ is log-concave, then \eqref{Cond2} is equivalent to
	\begin{equation}\label{liminf-condition}
	\liminf_{x\to\infty}\frac{\psi(x)}{x\psi'(x)}>0.
	\end{equation}
\end{lemma}

\begin{proof}
It is clear that \eqref{liminf-condition} implies \eqref{Cond2}. Conversely, suppose that \eqref{Cond2} holds, and suppose on the contrary to the assertion that there exists a strictly increasing sequence $\{x_n\}_n$ of positive real numbers tending to infinity such that
	$$
	\lim_{n\to\infty}\frac{\psi(x_n)}{x_n\psi'(x_n)}=0.
	$$
Let $\varepsilon>0$. Then there exists a positive integer $N$ such that $x_n\geq 1$ and
	$$
	0<\frac{\psi(x_n)}{x_n\psi'(x_n)}<\varepsilon,\quad n\geq N.
	$$
For every $n\in\N$ there exists a positive integer $k_n$ such that $k_n\leq x_n\leq k_n+1$. The assumption that $\psi$ is log-concave implies that $(\log \psi(x))'=\psi'(x)/\psi(x)$ is a non-increasing function. Therefore,
	$$
	\frac{1}{2k_n}\leq \frac{1}{k_n+1}\leq \frac{1}{x_n}\leq \varepsilon \frac{\psi'(x_n)}{\psi(x_n)}\leq \varepsilon\frac{\psi'(k_n)}{\psi(k_n)},\quad n\geq N,
	$$
and consequently
	$$
	\lim_{n\to\infty}\frac{\psi(k_n)}{k_n\psi'(k_n)}=0,
	$$
which violates \eqref{Cond2}.
\end{proof}

The order and the lower order of a function $\psi\in\D$ are given, respectively, by
	$$
	\overline{\rho}(\psi)=\limsup_{x\to\infty}\frac{\log\psi(x)}{\log x}
	\quad\textnormal{and}\quad
	\underline{\rho}(\psi)=\liminf_{x\to\infty}\frac{\log\psi(x)}{\log x}.
	$$

\begin{lemma}\label{observations-lemma}
Let $\psi\in\D_2$, and suppose that one of the following conditions holds.
\begin{itemize}
\item[\textnormal{(a)}] There exist constants  $C\geq 2$ and $R>0$ such that 
	\begin{equation}\label{doubling-condition}
	\psi(2x)\leq C\psi(x),\quad x\geq R.
	\end{equation}
\item[\textnormal{(b)}] $\psi$ satisfies
	\begin{equation}\label{psi-limsup-condition}
	\limsup_{x\to\infty}\left(x\log\frac{\psi(x+1)}{\psi(x)}\right)<\infty.
	\end{equation}
\end{itemize}
Then $\psi$ is of finite order of growth and
\eqref{liminf-condition} holds.	
\end{lemma}

\begin{proof}
Suppose  that \eqref{doubling-condition} holds. We have
	$$
	x\psi'(x)=\psi'(x)\int_x^{2x}dt\leq \int_x^{2x}\psi'(t)\, dt
	= \psi(2x)-\psi(x)\leq \psi(2x)\leq C\psi(x)
	$$
for all $x\geq R$. Thus the limit inferior in \eqref{liminf-condition} is $\geq 1/C$. Moreover, it follows by  \cite[Lemma~4.3]{HL} that $\overline{\rho}(\psi)\leq\log C/\log 2$. On the other hand, if \eqref{psi-limsup-condition} holds, then \cite[Theorem~4.5(a)]{HL} implies \eqref{doubling-condition}, from which the assertions follow.
\end{proof}

Each of the three conditions \eqref{doubling-condition}, \eqref{psi-limsup-condition} and \eqref{liminf-condition} is satisfied by every power function $\psi(x)=x^p$, $p\geq 1$.
Next, let's look at the converse of Lemma~\ref{observations-lemma}.

\begin{lemma}
If $\psi\in\D$ satisfies \eqref{liminf-condition}, then $\rho(\psi)<\infty$. In addition, if $\psi$ is log-concave, then \eqref{doubling-condition} and \eqref{psi-limsup-condition} hold.
\end{lemma}

\begin{proof}
Since \eqref{liminf-condition} holds, there exist constants  $C>0$ and $R>0$ such that
	$$
	\frac{\psi(x)}{x\psi'(x)}\geq \frac{1}{C},\quad x\geq R,
	$$
that is,
	$$
	\frac{\psi'(x)}{\psi(x)}\leq \frac{C}{x},\quad x\geq R.
	$$
Integrating both sides, we obtain $\psi(x)\leq \psi(R)\left(\frac{x}{R}\right)^C$ for $x\geq R$, which shows that $\psi$ is of finite order of growth. 

The assumption that $\psi$ is log-concave implies that $(\log \psi(x))'=\psi'(x)/\psi(x)$ is a non-increasing function. On the other hand, the assumption \eqref{liminf-condition} implies
	$$
	\limsup_{x\to\infty}\frac{x\psi'(x)}{\psi(x)}<\infty.
	$$
By combining these two properties, we obtain
	$$
	\log\frac{\psi(2x)}{\psi(x)}=\int_x^{2x}\frac{\psi'(t)}{\psi(t)}\, dt\leq \frac{x\psi'(x)}{\psi(x)}=O(1),
	$$
from which the assertion \eqref{doubling-condition} follows. On the other hand,
	$$
	x\log\frac{\psi(x+1)}{\psi(x)}=x\int_x^{x+1}\frac{\psi'(t)}{\psi(t)}\, dt
	\leq \frac{x\psi'(x)}{\psi(x)}=O(1),
	$$
from which the assertion \eqref{psi-limsup-condition} follows.
\end{proof}

\begin{remark}
(a) For $\psi\in\D$, the generalized L'Hospital's rule \cite{Taylor} gives
	$$
	\liminf_{x\to\infty}\frac{x\psi'(x)}{\psi(x)}\leq\underline{\rho}(\psi)\leq \overline{\rho}(\psi)\leq\limsup_{x\to\infty}\frac{x\psi'(x)}{\psi(x)}.
	 $$	
Hence, if the limit $\lim_{x\to\infty}\frac{x\psi'(x)}{\psi(x)}$ exists, then
	$$
	\underline{\rho}(\psi)=\overline{\rho}(\psi)=\lim_{x\to\infty}\frac{x\psi'(x)}{\psi(x)}.
	$$

(b) If $\psi:(0,\infty)\to (0,\infty)$ has a continuous derivative for all $x>0$ large enough, and if
	$$
	\lim_{x\to\infty}\frac{x\psi'(x)}{\psi(x)}=0,
	$$
then $\psi$ is slowly varying \cite[p.~7]{Seneta}.
\end{remark}


\section{Generalization of Gasull-Ma\~{n}osas' theorem:\\ Research question (Q5)}

Theorem~\ref{Th3} below generalizes Theorem~\ref{Gasull} and reduces to it when $\psi(x)=x$. This answers Research question (Q5).

\begin{theorem}\label{Th3}
	Let $\psi\in \D$, and let $\{c_n\}_n$ be a positive real sequence such that $\left\{\frac{c_n}{\psi'(n)}\right\}_n$ is non-increasing and tends to zero. If the sub-signed series \eqref{Sub-signed} converges, where $A$ and $B$ are defined in \eqref{Sets}, then
	\begin{eqnarray}\label{NC}
	\lim_{n\to\infty}\frac{A_{\psi}(n)-B_{\psi}(n)}{\psi'(n)}c_n=0.
	\end{eqnarray}
In addition, if $\psi$ satisfies \eqref{asym} and \eqref{Suf1},
then $\operatorname{d}_{\psi}(A)$ exists if and only if $\operatorname{d}_{\psi}(B)$ exists, and in this case, the two densities are equal.
\end{theorem}

Define $C:=\{n\in \mathbb{N}: m_n=0\}$. Then $\N=A\cup B\cup C$ in Theorem~\ref{Th3}. Under the assumptions of Theorem~\ref{Th3}, if one of $\operatorname{d}_{\psi}(A)$ or $\operatorname{d}_{\psi}(B)$ exists, then $\operatorname{d}_{\psi}(C)$ exists and $\operatorname{d}_{\psi}(C)=1-2\delta$, where $\delta\in [0,1/2]$ is a constant such that  
$\operatorname{d}_{\psi}(A)=\operatorname{d}_{\psi}(B)=\delta$. 

Theorem~\ref{Th3} generalizes Theorem~\ref{Th2}, since $B_{\psi}(n)=0$ holds if $m_n\in \{0,1\}$. Moreover, if \eqref{Sub-signed} represents a signed series, meaning $m_n\in \{-1,1\}$, then $A\cup B= \mathbb{N}$, resulting in
\begin{eqnarray}\label{A+B}
	A_{\psi}(n)+B_{\psi}(n)=\sum_{k=1}^n\psi'(k).
\end{eqnarray}
Under the additional assumption that $\psi$ satisfies \eqref{asym}, the assertion \eqref{NC} transforms into
\begin{eqnarray}\label{NC2}
	\frac{2A_{\psi}(n)-\psi(n)}{\psi'(n)}c_n\to 0,\quad n\to \infty.
\end{eqnarray}
These observations are summarized in the following corollary.

\begin{corollary}
Let $\psi \in \D$ satisfy \eqref{asym}, and let $\{c_n\}_n$ be a positive real sequence such that $\left\{\frac{c_n}{\psi'(n)}\right\}_n$ is non-increasing and tends to zero. If the subsets $A$ and $B$ defined in \eqref{Sets} satisfy $A\cup B=\mathbb{N}$ and \eqref{Sub-signed} converges, then \eqref{NC2} holds. Moreover, if \eqref{Suf1} holds, then $$\operatorname{d}_{\psi}(A)=\operatorname{d}_{\psi}(B)=\frac{1}{2}.$$ 
\end{corollary}

The proof of Theorem~\ref{Th3} is based on two auxiliary lemmas.

\begin{lemma}\label{Lem1}
Let $\{c_n\}_n$ be a non-vanishing real sequence, let $\{m_n\}_n$ be a sequence with $m_n \in\{-1,0,1\}$, and let $P_n=\sum_{k=1}^n m_k a_k$ be the corresponding partial sum. Then for any $\psi\in \mathcal{D}$, we have
	\begin{eqnarray}
	P_n&=&	\frac{A_{\psi}(n)-B_{\psi}(n)}{\psi'(n)}c_n+\frac{c_n}{\psi'(n)}\sum_{k=1}^{n-1}\left(\frac{\psi'(k+1)}{c_{k+1}}-\frac{\psi'(k)}{c_k} \right) P_k\label{id1}
	\end{eqnarray}
\end{lemma}

\begin{proof}
	We begin by expressing $m_1=\frac{P_1}{c_1}$ and $m_k=\frac{P_k-P_{k-1}}{c_k}$ for $k\geq 2$. Observe that
	$$
	A_{\psi}(n)=\sum_{k=1}^n\psi'(k)\chi_A(k)=\sum_{k=1}^n\frac{m_k(m_k+1)}{2}\psi'(k)
	$$
	and
	$$
	B_{\psi}(n)=\sum_{k=1}^n\psi'(k)\chi_B(k)=\sum_{k=1}^n\frac{m_k(m_k-1)}{2}\psi'(k).
	$$
	This yields
	\begin{eqnarray}
		A_{\psi}(n)-B_{\psi}(n)&=&\sum_{k=1}^nm_k\psi'(k)\nonumber \\
		&=&\frac{P_1}{c_1}\psi'(1)+\sum_{k=2}^n\frac{P_k-P_{k-1}}{c_k}\psi'(k)\nonumber\\
		&=&\frac{P_n}{c_n}\psi'(n)-\sum_{k=1}^{n-1}\left(\frac{\psi'(k+1)}{c_{k+1}}-\frac{\psi'(k)}{c_k} \right) P_k,\nonumber
	\end{eqnarray}
	which concludes the proof of \eqref{id1}. 
\end{proof}

The following lemma is a special case of Toeplitz's theorem, and it plays a key role in proving Theorem~\ref{Th3}.

\begin{lemma}  \label{Lem2}
\textnormal{\cite[Theorem 2.1]{GM}}
	Let $\{c_{n, m}\}_{(n,m)}$ be a double sequence satisfying the following properties.
	\begin{itemize}
		\item[\textnormal{(a)}] For all $(n, m) \in \mathbb{N} \times \mathbb{N}$, $c_{n, m} \geq 0$  and $c_{n, m}=0$ when $m>n$.
		\item[\textnormal{(b)}] $\lim\limits_{n\to \infty} \sum_{k=1}^{n}c_{n,k}=1$.
		\item[\textnormal{(c)}] For any fixed $m \in \mathbb{N}$, $\lim\limits_{n\to \infty} c_{n, m}=0$.
	\end{itemize}
Then, for any real sequence $\{x_n\}$, the sequence $\{y_n\}$ defined by $y_n=\sum_{k=1}^n c_{n, k} x_k$ satisfies
	$$
	\liminf_{n\to \infty} x_n \leq \liminf_{n\to \infty} y_n \leq \limsup_{n\to \infty} y_n \leq \limsup_{n\to \infty} x_n.
	$$
\end{lemma}

We are now ready for the proof of Theorem~\ref{Th3}.

\begin{proof}[Proof of Theorem~\ref{Th3}]
Recalling \eqref{id1} from Lemma~\ref{Lem1}, we have
\begin{eqnarray}\label{Id1'}
		\frac{A_{\psi}(n)-B_{\psi}(n)}{\psi'(n)}c_n=P_n-\frac{c_n}{\psi'(n)}\sum_{k=1}^{n-1}\left(\frac{\psi'(k+1)}{c_{k+1}}-\frac{\psi'(k)}{c_k} \right) P_k.
\end{eqnarray}
Set
	\[
	c_{n, k} = 
	\begin{cases}
		\frac{c_n}{\psi'(n)}\left( \frac{\psi'(k+1)}{c_{k+1}} - \frac{\psi'(k)}{c_k} \right), & \text{if } k < n, \\
		0, & \text{otherwise.}
	\end{cases}
	\]
	Clearly, $c_{n,k}\geq 0$ follows from the fact that $\left\{\frac{c_n}{\psi'(n)}\right\}_n $ is non-increasing. Moreover,
	$$
	\sum_{k=1}^{n-1}c_{n,k}=1-\frac{\psi'(1)}{c_1}\frac{c_n}{\psi'(n)}\to 1,\quad n\to \infty,
	$$
	and for any fixed $k$, we have $\lim\limits_{n\to \infty}c_{n,k}=0$. Thus, by applying Lemma~\ref{Lem2}, if $\{P_n\}_n$ converges to $P$, then
	$$
	\sum_{k=1}^{n} c_{n,k}P_k	=\frac{c_n}{\psi'(n)}\sum_{k=1}^{n-1}\left(\frac{\psi'(k+1)}{c_{k+1}}-\frac{\psi'(k)}{c_k} \right) P_k \to P,\quad n\to \infty.
	$$
This, together with \eqref{Id1'} yields \eqref{NC}. 

Assuming that \eqref{Suf1} holds, we may use Lemma~\ref{lem3}, with
	$$
	a_n=\frac{\psi(n)}{\psi'(n)}c_n \quad \text{and}\quad b_n=\frac{A_{\psi}(n)-B_{\psi}(n)}{\psi(n)}
	$$
to conclude that 
	$$
	\lim_{n\to\infty} \left( \frac{A_{\psi}(n)}{\psi(n)} -\frac{B_{\psi}(n)}{\psi(n)}\right) =0.
	$$
This completes the proof of Theorem~\ref{Th3} via assuming \eqref{asym}.
\end{proof}

The subsequent result extends Theorem~\ref{Salat2-thm} within the context of sub-signed series.

\begin{theorem}
	Let $\psi\in\D$ satisfy \eqref{asym}, and let $\{c_n\}_n$ be a non-negative real sequence satisfying \eqref{Salat-condition30}. If the sub-signed series \eqref{Sub-signed} converges, where $A$ and $B$ are defined in \eqref{Sets}, then the following assertions hold.
	\begin{itemize}
		\item[\textnormal{(a)}] If $\left\{ \frac{c_n}{\psi'(n)}\right\}_n $ is non-increasing, then
	\begin{eqnarray}\label{inq3}
		\liminf_{n\to\infty} \frac{A_{\psi}(n)-B_{\psi}(n)}{\psi(n)}=0.
	\end{eqnarray}
	In particular, if $A\cup B=\mathbb{N}$, then 	
	$$
\max \left\lbrace 	\underline{	\operatorname{d}}_{\psi}(A), 	\underline{	\operatorname{d}}_{\psi}(B) \right\rbrace \leq \frac{1}{2} \leq \min \left\lbrace \overline{	\operatorname{d}}_{\psi}(A), 	\overline{	\operatorname{d}}_{\psi}(B)\right\rbrace .
	$$
		\item[\textnormal{(b)}] If $\left\{ \frac{c_n}{\psi'(n)}\right\}_n $ is non-decreasing, then
	\begin{eqnarray}\label{inq4}
		\lim_{n\to\infty} \frac{A_{\psi}(n)-B_{\psi}(n)}{\psi(n)}=0.
	\end{eqnarray}
		In particular, if $A\cup B=\mathbb{N}$, then $\operatorname{d}_{\psi}(A)=\operatorname{d}_{\psi}(B)=\frac{1}{2}.$
	\end{itemize}
\end{theorem}

\begin{proof}
(a) Recall that
	$$
	\frac{A_{\psi}(n)-B_{\psi}(n)}{\sum_{k=1}^n\psi'(k)}=\frac{\sum_{k=1}^nm_k\psi'(k)}{\sum_{k=1}^n\psi'(k)}.
	$$
Choosing $a_n = c_n$, $b_n = \psi'(n)$ and $s_n = m_n$ in Theorem~\ref{L1}, we obtain
\begin{eqnarray}\label{Raj1}
	\liminf_{n\to\infty}\frac{\sum_{k=1}^nm_k\psi'(k)}{\sum_{k=1}^n\psi'(k)} &\leq& \liminf_{n\to\infty}\frac{\sum_{k=1}^nm_kc_k}{\sum_{k=1}^nc_k} \nonumber\\
	&\leq& \limsup_{n\to\infty}\frac{\sum_{k=1}^nm_kc_k}{\sum_{k=1}^nc_k} \leq \limsup_{n\to\infty}\frac{\sum_{k=1}^nm_k\psi'(k)}{\sum_{k=1}^n\psi'(k)}.
\end{eqnarray}
Hence the assertion \eqref{inq3} follows from \eqref{asym}, \eqref{Raj1} and
	\begin{eqnarray}\label{lim1}
	\lim_{n\to\infty}\frac{\sum_{k=1}^nm_kc_k}{\sum_{k=1}^nc_k}=0.
	\end{eqnarray}
The second assertion in (a) follows from \eqref{asym}, \eqref{A+B} and \eqref{Raj1}.

(b) Suppose now that $\left\{\frac{c_n}{\psi'(n)}\right\}_n$ is non-decreasing, in which case $\left\{\frac{\psi'(n)}{c_n}\right\}_n$ is non-increasing. Using Theorem~\ref{L1} again, we obtain
	\begin{eqnarray*}\label{Raj2}
	\liminf_{n\to\infty}\frac{\sum_{k=1}^nm_kc_k}{\sum_{k=1}^nc_k}	 &\leq& \liminf_{n\to\infty}\frac{\sum_{k=1}^nm_k\psi'(k)}{\sum_{k=1}^n\psi'(k)} \nonumber\\
	&\leq&  \limsup_{n\to\infty}\frac{\sum_{k=1}^nm_k\psi'(k)}{\sum_{k=1}^n\psi'(k)}\leq\limsup_{n\to\infty}\frac{\sum_{k=1}^nm_kc_k}{\sum_{k=1}^nc_k} .
	\end{eqnarray*}
Then, by making use of \eqref{lim1} again, the assertion \eqref{inq4} follows. The second assertion in (b) follows from \eqref{asym}, \eqref{A+B} and \eqref{inq4}.
\end{proof}

\appendix


\section{Proofs of Propositions~\ref{Salat-example} and \ref{sharpness-prop}}\label{App-A}

\noindent
\emph{Proof of Proposition~\ref{Salat-example}}. For $n\in\N$, define
	$$
	c_{n^{n}+k(n)}:=\frac{1}{n^{n+2}},
	$$
where $k(n)$ is an integer satisfying $0 \leq k(n)\leq (n+1)^{n+1}-n^n-1$. That is, 
	$$
	c_1=c_2=c_3=1, \quad c_4=\ldots=c_{26}=\frac{1}{2^4}, \quad  c_{27}=\ldots=c_{255}=\frac{1}{3^5},
	$$
and so on. Therefore,
	$$
	\sum_{n=1}^{\infty}c_n=\sum_{n=1}^{\infty} \sum_{k=0}^{\Delta n^n -1} c_{n^n+k}=\sum_{n=1}^{\infty}\frac{(n+1)^{n+1}-n^n}{n^{n+2}}.
	$$
The series $\sum_{n=1}^{\infty} c_n$ diverges, due to the fact that
	\begin{equation}\label{not-right}
	(n+1)^{n+1}-n^n>(e-1) n^{n+1},\quad n\in\N.
	\end{equation}
To see that \eqref{not-right} holds, we write
	$$
	\frac{(n+1)^{n+1}-n^n}{n^{n+1}}=\left( 1+\frac{1}{n}\right) ^{n+1}-\frac{1}{n}\geq \left( 1+\frac{1}{n}\right) ^{n+1}-1.
	$$
This together with the fact
	$
	\left(1+\frac{1}{n}\right)^n<e<\left(1+\frac{1}{n}\right)^{n+1}
	$
result in \eqref{not-right}.

Given $\delta\in (0,1)$, let $a=\frac{1}{1-\delta}$, and define the set $A\subset\N$ in the assertion by
	$$
	A:=\{n^n+k(n):0\leq k(n)\leq (a-1)n^n\}
	=\N\cap \bigcup_{n\geq 1} \left[n^n,an^n \right].
	$$  
Using \eqref{not-right}, we see that $(n+1)^{n+1}-n^n>n^{n+1}\geq (a-1)n^n$ for every $n\geq a-1$, from which $an^n<(n+1)^{n+1}$ for every $n\geq a-1$. Hence the intervals $\left[n^n,an^n \right]$ are pairwise disjoint for every $n\geq m:=\lceil a-1\rceil$.
Moreover,
	\begin{equation*}
	\begin{split}
	\sum_{n=1}^{\infty}	\chi_A(n)c_n&=O(1)+\sum_{n=m}^{\infty}	\chi_A(n)c_n
	=O(1)+\sum_{n=m}^{\infty} \sum_{k=0}^{\Delta n^n -1}\chi_A(n^n+k)c_{n^n+k}\\
	&=O(1)+\sum_{n=m}^{\infty}\frac{(a-1)n^n+1}{n^{n+2}}
	\leq O(1)+a \sum_{n=1}^{\infty}\frac{1}{n^2}<\infty.
	\end{split}	
	\end{equation*}
Thus it remains to prove that $\overline{\operatorname{d}}(A)\geq \delta$ and $\underline{\operatorname{d}}(A)=0.$
To this end, we have
	\begin{equation*}
	A(an^n)\geq\sum_{k=m}^{n}\sum_{j=k^k}^{\lfloor ak^k\rfloor}1\geq\sum_{k=m}^{n}(a-1)k^k\geq (a-1)n^n,
	\end{equation*}
and consequently
	$$
	\frac{A(an^n)}{an^n}\geq \frac{a-1}{a}=\delta\ \implies \ \overline{\operatorname{d}}(A)\geq \delta.
	$$
Meanwhile,
	\begin{eqnarray*}
	A((n+1)^{n+1}-1)
&\leq&	\sum_{k=1}^{n}\sum_{j=k^k}^{\lceil ak^k\rceil}1\leq\sum_{k=1}^{n}((a-1)k^k+2)
\leq a\sum_{k=1}^{n}k^k+2n\\
&\leq & a\sum_{k=1}^n n^k +2n=an\cdot\frac{n^n-1}{n-1}+2n\leq 2an^{n}+2n,
	\end{eqnarray*}
and consequently
	$$
	\frac{A((n+1)^{n+1}-1)}{(n+1)^{n+1}-1} \leq \frac{2an^n+2n}{(n+1)^{n+1}-1}\sim\frac{2a}{e(n+1)}\ \implies\  \underline{\operatorname{d}}(A)=0.
	$$
This completes the proof of Proposition~\ref{Salat-example}. \hfill$\Box$

\medskip
\noindent
\emph{Proof of Proposition~\ref{sharpness-prop}}.
Let $\delta\in (0,1)$, and let $\{c_n\}_n$ and $A\subset\N$ be as in the proof of Proposition~\ref{Salat-example}. Define $\psi(x)=\frac{x}{\log\log (x+e^e)}$. Since
	$$
	\psi'(x)=\frac{\log\log(x+e^e)-\frac{x}{(x+e^e)\log(x+e^e)}}{(\log\log(x+e^e))^2},
	$$
it is easy to see that $\psi\in\D_1$.
We need to prove that
	$$
	\sum_{n=1}^\infty \psi'(n)c_n=\infty,\quad \sum_{n=1}^{\infty}\chi_A(n)c_n<\infty,\quad \overline{\operatorname{d}}_\psi(A)\geq \delta.
	$$
For a fixed $\alpha>1$, the equation
	$$
	\log\log\left(x+e^e\right)-1=\frac{1}{\alpha}\log\log\left(x+e^e\right)
	$$
has a unique solution $x_0:=\exp\left(e^\frac{\alpha}{\alpha-1}-e^e\right)$. We have
	\begin{equation}\label{psi-prime-down}
	\psi'(x)>\frac{\log\log(x+e^e)-1}{(\log\log(x+e^e))^2}
	>\left\{
\begin{array}{ll}
	0,\ & x>0,\\
	\frac{1}{\alpha\log\log(x+e^e)},\ & x>x_0,\end{array}\right.
	\end{equation}
and
	$$
	\sum_{n=1}^{\infty}\psi'(n)c_n 
	= \sum_{n=1}^{\infty} \sum_{k=0}^{\Delta n^n -1}\psi'(n^n+k)c_{n^n+k}
	=\sum_{n=1}^{\infty} \sum_{k=0}^{\Delta n^n -1}\frac{\psi'(n^n+k)}{n^{n+2}}. 
	$$
Using \eqref{not-right}, \eqref{psi-prime-down} and the fact that $\psi'$ is non-increasing, we obtain
	\begin{eqnarray*}
	\sum_{n=1}^{\infty}\psi'(n)c_n&\geq& (e-1)\sum_{n=1}^{\infty}\frac{\psi'((n+1)^{n+1})}{n}\\
	&\geq& \frac{e-1}{\alpha}\sum_{n\geq M}^{\infty}\frac{1}{n\log\log((n+1)^{n+1}+e^e)},
	\end{eqnarray*}
where $M$ is an integer satisfying $(M+1)^{M+1}>x_0$. We see that $\sum_{n=1}^{\infty}\psi'(n)c_n$ is divergent. Further, since $\{c_n\}_n$ and $A\subset\N$ are the same as in the proof of Proposition~\ref{Salat-example},
we have $\sum_{n=1}^{\infty}\chi_A(n)c_n<\infty$. 

It remains to prove that $\overline{\operatorname{d}}_\psi(A)\geq \delta$. Let $a:=\frac{1}{1-\delta}$, and let $m\geq \lceil a-1\rceil$ be an integer satisfying $m^m\geq x_0$.
Similarly as in the proof of Proposition~\ref{Salat-example}, for $n\geq m$, we obtain
	\begin{eqnarray*}
	A_\psi(an^n)&\geq &\sum_{k=m}^{\lfloor an^n\rfloor}\psi'(k)\chi_A(k)=\sum_{k=m}^{n}\sum_{j=k^k}^{\lfloor ak^k\rfloor}\psi'(j)\\
	&\geq& \sum_{k=1}^{n}\frac{(a-1)k^k}{\alpha\log\log(\lfloor ak^k\rfloor+e^e)}
	\geq \frac{(a-1)n^n}{\alpha\log\log( an^n+e^e)}.
	\end{eqnarray*}
Consequently, for $n\geq m$,
	$$
	\frac{A_\psi(an^n)}{\psi(an^n)}\geq \frac{a-1}{\alpha a}=\frac{\delta}{\alpha}\ \implies \ \overline{\operatorname{d}}_\psi(A)\geq \frac{\delta}{\alpha}.
	$$
Letting $\alpha\to 1^+$, we get $\overline{\operatorname{d}}_\psi(A)\geq \delta$. This completes the proof of Proposition~\ref{sharpness-prop}.\hfill$\Box$


\section{Second proof of Theorem~\ref{Salat2-thm}.}\label{App-B}

(a) We prove the assumption by assuming that $\psi\in\D$ is log-concave and satisfies \eqref{asym}. It suffices to prove that if
	\begin{equation}\label{positive-upper-psi}
	\underline{\operatorname{d}}_{\psi}(A)=\liminf_{n\to\infty}\frac{A_\psi(n)}{\psi(n)}>0,
	\end{equation}
then the series in \eqref{subseries} diverges. Hence, suppose that \eqref{positive-upper-psi} holds. Then there exist constants $\delta>0$ and $N\in \N$ such that
	$$
	\sum_{k=1}^n\psi'(k)\chi_A(k)=A_\psi(n)>\delta\psi(n),\quad n\geq N.
	$$
By Abel's partial summation formula \cite[p.~194]{Apostol}, and by the assumption that the sequence $\left\{\frac{c_n}{\psi'(n)}\right\}_n$ is non-increasing, we obtain 
	\begin{eqnarray*}
	\sum_{k=1}^n \chi_A(k)c_k 
	&=& \sum_{k=1}^n\psi'(k)\chi_A(k)\frac{c_k}{\psi'(k)}\\
	&\overset{\text{Abel}}{=}& A_\psi(n)\frac{c_{n+1}}{\psi'(n+1)}+\sum_{k=1}^nA_\psi(k)\left(\frac{c_k}{\psi'(k)}-\frac{c_{k+1}}{\psi'(k+1)}\right)\\
	&\geq& \delta\psi(n)\frac{c_{n+1}}{\psi'(n+1)}+\delta\sum_{k=N}^n\psi(k)\left(\frac{c_k}{\psi'(k)}-\frac{c_{k+1}}{\psi'(k+1)}\right)\\
	&=& \delta\psi(n)\frac{c_{n+1}}{\psi'(n+1)}+\delta\sum_{k=N}^n\psi(k)\frac{c_k}{\psi'(k)}-\delta\sum_{k=N+1}^{n+1}\psi(k-1)\frac{c_k}{\psi'(k)}\\
	&=& \delta\psi(N)\frac{c_N}{\psi'(N)}+\delta\sum_{k=N+1}^n\frac{\psi(k)-\psi(k-1)}{\psi'(k)}c_k\\
	&\geq & \delta\sum_{k=N+1}^n\frac{\psi(k)-\psi(k-1)}{\psi'(k)}c_k,\quad n\geq N+1.
	\end{eqnarray*}
The fact that $\sum_{k=1}^n \chi_A(k)c_k$ diverges follows by the assumption \eqref{Salat-condition30} and the following lemma.

\begin{lemma}\label{log-concave-lemma}
If $\psi\in\D$ is log-concave, then
	$$
	\inf_{x>1}\frac{\psi(x)-\psi(x-1)}{\psi'(x)}>0.
	$$
\end{lemma}

\begin{proof}
The assertion is equivalent to
	\begin{equation}\label{sup-bounded}
	\sup_{x>1}\frac{\psi'(x)}{\psi(x)-\psi(x-1)}<\infty.
	\end{equation}
Since $\psi'(x)/\psi(x)$ is non-increasing, it is bounded, say by $M>0$. Then
	$$
	\log\frac{\psi(x)}{\psi(x-1)}=\int_{x-1}^x\frac{\psi'(t)}{\psi(t)}\, dt\leq \frac{\psi'(x-1)}{\psi(x-1)}\leq M.
	$$
By the mean value theorem, there exists a constant $c\in (x-1,x)$ such that $\psi'(c)=\psi(x)-\psi(x-1)$. Then
	$$
	\frac{\psi'(x)}{\psi(x)-\psi(x-1)}=\frac{\psi'(x)}{\psi(c)}\cdot\frac{\psi(c)}{\psi'(c)}\leq \frac{\psi'(x)}{\psi(x-1)}\cdot\frac{\psi(x)}{\psi'(x)}=\frac{\psi(x)}{\psi(x-1)}\leq e^M,
	$$
which implies \eqref{sup-bounded}.
\end{proof}

The assertion of Lemma~\ref{log-concave-lemma} is satisfied by the functions $\psi(x)=e^{\alpha x}$, $\alpha>0$, which have the maximal growth-rate in the class of log-concave functions. 

\medskip
(b) We prove, without assuming $\log$-concavity of $\psi$, that if
	$$
	\overline{\operatorname{d}}_{\psi}(A)=\limsup_{n\to\infty}\frac{A_\psi(n)}{\psi(n)}>0,
	$$
then the series in \eqref{convergent-sum0} diverges.
It is clear that $A_\psi(n)$ is non-decreasing, and by the assumption $\overline{\operatorname{d}}_{\psi}(A)>0$, it is also an unbounded function of $n$. By Abel's partial summation, and by the assumption that  $\left\{\frac{c_n}{\psi'(n)}\right\}_n$ is non-decreasing, we obtain
	\begin{eqnarray*}
	\sum_{k=1}^n \chi_A(k)c_k 
	&=& \sum_{k=1}^n\psi'(k)\chi_A(k)\frac{c_k}{\psi'(k)}\\
	&\overset{\text{Abel}}{=}& A_\psi(n)\frac{c_{n}}{\psi'(n)}+\sum_{k=1}^{n-1}A_\psi(k)\left(\frac{c_{k}}{\psi'(k)}-\frac{c_{k+1}}{\psi'(k+1)}\right)\\
	&\geq & A_\psi(n)\frac{c_{n}}{\psi'(n)}-A_\psi(n-1)\sum_{k=1}^{n-1}\left(\frac{c_{k+1}}{\psi'(k+1)}-\frac{c_{k}}{\psi'(k)}\right)\\
	&\geq& A_\psi(n)\frac{c_{n}}{\psi'(n)}-A_\psi(n)\left(\frac{c_n}{\psi'(n)}-\frac{c_1}{\psi'(1)}\right)= A_\psi(n)\frac{c_1}{\psi'(1)},
	\end{eqnarray*}
where the lower bound tends to $\infty$ as $n\to\infty$. Note that $c_1>0$ can be assumed here, for if $c_1=0$, then we can start the summation in the reasoning above from $k=m$, where $c_m>0$.


\bigskip
E-mail: \texttt{janne.heittokangas@uef.fi}

\textsc{University of Eastern Finland, Department of Physics and Mathematics, P.O.~Box 111, 80100 Joensuu, Finland}

\medskip
E-mail: \texttt{z.latreuch@nhsm.edu.dz}

\textsc{National Higher School of Mathematics, Scientific and Technology Hub of Sidi Abdellah, P.O.~Box 68, Algiers 16093, Algeria}

\end{document}